\documentclass[12pt]{amsart}
\usepackage{amsfonts,amssymb,amscd,amsmath,enumerate,verbatim}
\usepackage[latin1]{inputenc}
\usepackage{amscd}
\usepackage{latexsym}
\usepackage{mathptmx}
\usepackage{multicol}
\input xy
\xyoption{all}
\def\NZQ{\mathbb}               

\def\ZZ{{\NZQ Z}}
\def\RR{{\NZQ R}}

\def\frk{\mathfrak}               

\def\Phi{{\frk N}}
\def\eb{{\bold e}}

\def\opn#1#2{\def#1{\operatorname{#2}}} 
\opn\chara{char} \opn\length{\ell} \opn\pd{pd} \opn\rk{rk}
\opn\projdim{proj\,dim} \opn\injdim{inj\,dim} \opn\rank{rank}
\opn\depth{depth} \opn\grade{grade} \opn\height{height}
\opn\embdim{emb\,dim} \opn\codim{codim}

\opn\Tr{Tr} \opn\bigrank{big\,rank}
\opn\superheight{superheight}\opn\lcm{lcm}
\opn\trdeg{tr\,deg}
\opn\reg{reg} \opn\lreg{lreg} \opn\ini{in} \opn\lpd{lpd}
\opn\size{size}\opn{\mult}{mult}
\opn\div{div} \opn\Div{Div} \opn\cl{cl} \opn\Cl{Cl}
\opn\Spec{Spec} \opn\Supp{Supp} \opn\supp{supp} \opn\Sing{Sing}
\opn\Ass{Ass} \opn\Min{Min}
\opn\Ann{Ann} \opn\Rad{Rad} \opn\Soc{Soc}
\opn\Syz{Syz} \opn\Im{Im} \opn\Ker{Ker} \opn\Coker{Coker}
\opn\Am{Am} \opn\Hom{Hom} \opn\Tor{Tor} \opn\Ext{Ext}
\opn\End{End} \opn\Aut{Aut} \opn\id{id} \opn\ini{in}

\opn\nat{nat}
\opn\pff{pf}
\opn\Pf{Pf} \opn\GL{GL} \opn\SL{SL} \opn\mod{mod} \opn\ord{ord}
\opn\Gin{Gin}
\opn\Hilb{Hilb}\opn\adeg{adeg}\opn\std{std}\opn\ip{infpt}
\opn\Pol{Pol}
\opn\sat{sat}
\opn\Var{Var}
\opn\Gen{Gen}

\opn\aff{aff} \opn\con{conv} \opn\relint{relint} \opn\st{st}
\opn\lk{lk} \opn\cn{cn} \opn\core{core} \opn\vol{vol}
\opn\link{link} \opn\star{star}
\opn\gr{gr}

\def\Cc{{\mathcal C}}
\def\Oc{{\mathcal O}}

\def\Jc{{\mathcal J}}

\def\Fc{{\mathcal F}}
\def\Pc{{\mathcal P}}
\def\Qc{{\mathcal Q}}
\def\pot#1#2{#1[\kern-0.28ex[#2]\kern-0.28ex]}

\opn\dirlim{\underrightarrow{\lim}}
\opn\inivlim{\underleftarrow{\lim}}
\let\to=\rightarrow

\def\Implies{\ifmmode\Longrightarrow \else
        \unskip${}\Longrightarrow{}$\ignorespaces\fi}
\def\implies{\ifmmode\Rightarrow \else
        \unskip${}\Rightarrow{}$\ignorespaces\fi}
\def\iff{\ifmmode\Longleftrightarrow \else
        \unskip${}\Longleftrightarrow{}$\ignorespaces\fi}

\let\:=\colon
\newtheorem{Theorem}{Theorem}[section]
\newtheorem{Lemma}[Theorem]{Lemma}
\newtheorem{Corollary}[Theorem]{Corollary}

\newtheorem{Example}[Theorem]{Example}

\newtheorem{Conjecture}[Theorem]{Conjecture}

\let\epsilon\varepsilon
\let\phi=\varphi
\let\kappa=\varkappa
\def\qed{\ifhmode\textqed\fi
      \ifmmode\ifinner\quad\qedsymbol\else\dispqed\fi\fi}
\def\textqed{\unskip\nobreak\penalty50
       \hskip2em\hbox{}\nobreak\hfil\qedsymbol
       \parfillskip=0pt \finalhyphendemerits=0}
\def\dispqed{\rlap{\qquad\qedsymbol}}

\opn\dis{dis}
\def\pnt{{\raise0.5mm\hbox{\large\bf.}}}

\opn\Lex{Lex}

\begin{document}
\title{Unimodular equivalence of order and chain polytopes}
\author{Takayuki Hibi and Nan Li}
\thanks{}
\subjclass{}
\address{Takayuki Hibi,
Department of Pure and Applied Mathematics,
Graduate School of Information Science and Technology,
Osaka University,
Toyonaka, Osaka 560-0043, Japan}
\email{hibi@math.sci.osaka-u.ac.jp}
\address{Nan Li,
Department of Mathematics,
Massachusetts Institute of Technology,
Cambridge, MA 02139, USA}
\email{nan@math.mit.edu}
\thanks{}
\begin{abstract}
The problem when the order polytope and the chain polytope of a finite partially ordered set
are unimodularly equivalent will be solved.
\end{abstract}
\subjclass{}
\thanks{
{\bf 2010 Mathematics Subject Classification:}
Primary 52B05; Secondary 06A07. \\
\hspace{5.3mm}{\bf Key words and phrases:}
chain polytope, order polytope, partially ordered set.}
\maketitle
\section*{Introduction}
The combinatorial structure of the order polytope $\Oc(P)$ and the chain polytope $\Cc(P)$
of a finite poset $P$
(partially ordered set) is discussed in \cite{Stanley}.
On the other hand, in \cite{Hibi} and \cite{TN}, it is shown that the toric ring of
each of $\Oc(P)$ and $\Cc(P)$ is an algebra with straightening laws (\cite[p. 124]{HibiRedBook})
on the distributive lattice
$L = \Jc(P)$, where $\Jc(P)$ is the set of all poset ideals of $P$, ordered by inclusion.
In the present paper,
the problem when $\Oc(P)$ and $\Cc(P)$ are unimodularly equivalent will be solved.

\section{The number of facets of chain and order polytopes}
Let $P = \{x_{1}, \ldots, x_{d}\}$ be a finite poset.
For each subset $W \subset P$, we associate $\rho(W) = \sum_{i \in W}\eb_{i} \in \RR^{d}$,
where $\eb_{1}, \ldots, \eb_{d}$ are the unit coordinate vectors of $\RR^{d}$.
In particular $\rho(\emptyset)$ is the origin of $\RR^{d}$.
A {\em poset ideal} of $P$ is a subset $I$ of $P$ such that,
for all $x_{i}$ and $x_{j}$ with
$x_{i} \in I$ and $x_{j} \leq x_{i}$, one has $x_{j} \in I$.
An {\em antichain} of $P$ is a subset
$A$ of $P$ such that $x_{i}$ and $x_{j}$ belonging to $A$ with $i \neq j$ are incomparable.
We say that $x_{j}$ {\em covers} $x_{i}$ if $x_{i} < x_{j}$ and
$x_{i} < x_{k} < x_{j}$ for no $x_{k} \in P$.
A chain $x_{j_{1}} < x_{j_{2}} < \cdots < x_{j_{\ell}}$ of $P$ is called
{\em saturated} if $x_{j_{q}}$ covers $x_{j_{q-1}}$ for $1 < q \leq \ell$.

Recall that the {\em order polytope} of $P$ is the convex polytope $\Oc(P) \subset \RR^{d}$
which consists of those $(a_{1}, \ldots, a_{d}) \in \RR^{d}$ such that
$0 \leq a_{i} \leq 1$ for every $1 \leq i \leq d$ together with
\[
a_{i} \geq a_{j}
\]
if $x_{i} \leq x_{j}$ in $P$.
The {\em chain polytope} of $P$ is the convex polytope $\Cc(P) \subset \RR^{d}$
which consists of those $(a_{1}, \ldots, a_{d}) \in \RR^{d}$ such that
$a_{i} \geq 0$ for every $1 \leq i \leq d$ together with
\[
a_{i_{1}} + a_{i_{2}} + \cdots + a_{i_{k}} \leq 1
\]
for every maximal chain $x_{i_{1}} < x_{i_{2}} < \cdots < x_{i_{k}}$ of $P$.

One has $\dim \Oc(P) = \dim \Cc(P) = d$.
The vertices of $\Oc(P)$ is
those $\rho(I)$ such that $I$ is a poset ideal of $P$
(\cite[Corollary 1.3]{Stanley})
and the vertices of $\Cc(P)$ is
those $\rho(A)$ such that $A$ is an antichain of $P$
(\cite[Theorem 2.2]{Stanley}).
In particular the number of vertices of $\Oc(P)$ is equal to that of $\Cc(P)$.
Moreover, the volume of $\Oc(P)$ and that of $\Cc(P)$ are equal to $e(P)/d!$, where
$e(P)$ is the number of linear extensions of $P$ (\cite[Corollary 4.2]{Stanley}).

Let $m_{*}(P)$ (resp. $m^{*}(P)$) denote the number of minimal (reps. maximal)
elements of $P$ and $h(P)$ the number of edges
of the Hasse diagram (\cite[p. 243]{StanleyEC}) of $P$.
In other words, $h(P)$ is the number of pairs $(x_{i}, x_{j}) \in P \times P$
such that $x_{j}$ covers $x_{i}$.
Let $c(P)$ denote the number of maximal chains of $P$.
It then follows immediately that

\begin{Lemma}
\label{numberoffacet}
The number of facets of $\Oc(P)$ is $m_{*}(P) + m^{*}(P) + h(P)$ and that
of $\Cc(P)$ is $d + c(P)$.
\end{Lemma}

\begin{Corollary}
\label{facetinequality}
The number of facets of $\Oc(P)$ is less than or equal to that of $\Cc(P)$.
\end{Corollary}

\begin{proof}
We work with induction on $d$, the number of elements of $P$.
Choose a minimal element $\alpha$ of $P$ which is not maximal.
One can assume
\begin{eqnarray}
\label{Boston}
m_{*}(P \setminus \{ \alpha \}) + m^{*}(P \setminus \{ \alpha \}) + h(P \setminus \{ \alpha \})
\leq (d - 1) + c(P \setminus \{ \alpha \}).
\end{eqnarray}
Let $\beta_{1}, \ldots, \beta_{s}, \gamma_{1}, \ldots, \gamma_{t}$ be the elements of $P$
which cover $\alpha$ such that each $\beta_{i}$ covers at least two elements of $P$
and each $\gamma_{j}$ covers no element of $P$ except for $\alpha$.
Let $N_{i}$ denote the number of saturated chains of the form
$\beta_{i} < x_{j_{1}}< x_{j_{2}} < \cdots$.  Then
\begin{eqnarray*}
m_{*}(P \setminus \{ \alpha \}) & = & m_{*}(P) - 1 + t; \\
m^{*}(P \setminus \{ \alpha \}) & = & m^{*}(P); \\
h(P \setminus \{ \alpha \}) & = & h(P) - (s + t); \\
c(P \setminus \{ \alpha \}) & = & c(P) - \sum_{i=1}^{s} N_{i}.
\end{eqnarray*}
Hence
\begin{eqnarray}
\label{Sapporo}
c(P \setminus \{ \alpha \}) \leq c(P) - s.
\end{eqnarray}
One has
\[
m_{*}(P \setminus \{ \alpha \}) + m^{*}(P \setminus \{ \alpha \}) + h(P \setminus \{ \alpha \})
= m_{*}(P) + m^{*}(P) + h(P) - (s + 1)
\]
and
\[
d - 1 + c(P \setminus \{ \alpha \}) \leq d - 1 + c(P) - s = d + c(P) - (s + 1).
\]
Thus, by virtue of the inequality $(\ref{Boston})$, it follows that
\begin{eqnarray*}
\label{Sydney}
m_{*}(P) + m^{*}(P) + h(P) \leq d + c(P),
\end{eqnarray*}
as desired.
\, \, \, \, \, \, \, \, \, \, \, \, \, \, \, \, \, \, \, \,
\, \, \, \, \, \, \, \, \, \, \, \, \, \, \, \, \, \, \, \,
\, \, \, \, \, \, \, \, \, \, \, \, \, \, \, \, \, \, \, \,
\end{proof}

We now come to a combinatorial characterization of $P$ for which the number of facets of
$\Oc(P)$ is equal to that of $\Cc(P)$.

\begin{Theorem}
\label{characterization}
The number of facets of $\Oc(P)$ is equal to that of $\Cc(P)$
if and only if the following poset
\bigskip
$$
\xy 0;/r.2pc/: (0,0)*{\circ}="a";
 (8,10)*{\circ}="b";
 (-8,10)*{\circ}="c";
 (-8,-10)*{\circ}="d";
 (8,-10)*{\circ}="e";
  "a"; "b"**\dir{-};
   "a"; "c"**\dir{-};
    "a"; "d"**\dir{-};
     "a"; "e"**\dir{-};
\endxy
$$

\smallskip

\hspace{6.3cm}
{\rm \bf \small Figure\,\,1}

\bigskip
\smallskip

\noindent
does not appear as a subposet {\em (}\cite[p. 243]{StanleyEC}{\em )} of $P$.
\end{Theorem}

\begin{proof}
The number of facets of $\Oc(P)$ is equal to that of $\Cc(P)$
if and only if, in the proof of Corollary \ref{facetinequality},
each of the inequalities $(\ref{Boston})$ and $(\ref{Sapporo})$
is an equality.

{\bf (``If'')}
Suppose that the poset of Figure $1$ does not appear as a subposet of $P$.
Then, in the proof of Corollary \ref{facetinequality},
one has $N_{i} = 1$ for $1 \leq i \leq s$.
Hence the inequality $(\ref{Sapporo})$ is an equality.
Moreover, the induction hypothesis guarantees that
the inequalities $(\ref{Boston})$ is an equality.
Thus the number of facets of $\Oc(P)$ is equal to that of $\Cc(P)$,
as required.

{\bf (``Only if'')}
Suppose that the poset of Figure $1$ appears as a subposet of $P$.
It then follows easily that
there exist $\delta, \xi, \mu, \phi, \psi$ of $P$ such that
(i) $\delta$ covers $\xi$ and $\mu$, (ii) $\delta < \phi$, $\delta < \psi$, and
(iii) $\phi$ and $\psi$ are incomparable.
\begin{itemize}
\item
If neither $\xi$ nor $\mu$ is a minimal element of $P$, then
the poset of Figure $1$ appears as a subposet of $P \setminus \{ \alpha \}$,
where $\alpha$ is any minimal element of $P$.  Hence
the induction hypothesis guarantees that
the inequality (\ref{Boston}) cannot be an equality.
\item
If either $\xi$ or $\mu$ coincides with a minimal element $\alpha$
of $P$, then, in the proof of Corollary \ref{facetinequality},
one has $N_{i} > 1$ for some $1 \leq i \leq s$.
Hence, the inequality (\ref{Sapporo}) cannot be an equality.
\end{itemize}
Hence, at least one of the inequalities (\ref{Boston}) and (\ref{Sapporo}) cannot be an equality.
Thus the number of facets of $\Oc(P)$ is less than that of $\Cc(P)$.
\, \, \, \, \, \, \, \, \, \, \, \, \, \, \, \, \, \, \, \,
\, \, \, \, \, \, \, \, \, \, \, \, \,
\end{proof}

\section{Unimodular equivalence}
Let $\ZZ^{d \times d}$ denote the set of $d \times d$ integral matrices.
Recall that a matrix $A \in \ZZ^{d \times d}$ is {\em unimodular} if $\det (A) = \pm 1$.
Given integral polytopes $\Pc \subset \RR^{d}$ of dimension $d$ and $\Qc \subset \RR^{d}$
of dimension $d$,
we say that $\Pc$ and $\Qc$ are {\em unimodularly equivalent}
if there exist a unimodular matrix $U \in \ZZ^{d \times d}$
and an integral vector ${\bf w} \in \ZZ^{d}$
such that $\Qc=f_U(\Pc)+{\bf w}$,
where $f_U$ is the linear transformation of $\RR^d$ defined by $U$,
i.e., $f_U({\bf v}) = {\bf v} \, U$ for all ${\bf v} \in \RR^d$.

Now, we wish to solve our pending problem when $\Oc(P)$ and $\Cc(P)$
are unimodularly equivalent.

\begin{Theorem}
\label{unimodular}
The order polytope $\Oc(P)$ and the chain polytope $\Cc(P)$ of a finite poset $P$
are unimodularly equivalent
if and only if the poset of Figure $1$ of Theorem \ref{characterization}
does not appear as a subposet of $P$.
\end{Theorem}

\begin{proof}
{\bf (``Only if'')}
If $\Oc(P)$ and $\Cc(P)$ are unimodularly equivalent, then in particular the number of facets
of $\Oc(P)$ and that of $\Cc(P)$ coincides.  Hence by virtue of Theorem \ref{characterization}
the poset of Figure $1$ does not appear as a subposet of $P$.

{\bf (``If'')}
Let $P = \{x_{1}, \ldots, x_{d}\}$ and
suppose that the poset of Figure $1$ does not appear as a subposet of $P$.
Fix $x \in P$ which is neither minimal nor maximal.
Then at least one of the following conditions are satisfied:
\begin{itemize}
\item
there is a unique saturated chain of the form $x = x_{i_{0}} > x_{i_{1}} > \cdots > x_{i_{k}}$,
where $x_{i_{k}}$ is a minimal element of $P$;
\item
there is a unique saturated chain of the form $x = x_{j_{0}} < x_{j_{1}} < \cdots < x_{j_{\ell}}$,
where $x_{i_{\ell}}$ is a maximal element of $P$.
\end{itemize}
Now, identifying $x_{1}, \ldots, x_{d}$ with the coordinates of $\RR^{d}$,
we introduce the affine map $\Psi \, : \RR^{d} \to \RR^{d}$ defined as follows:
\begin{itemize}
\item
$\Psi(x_{i}) = 1 - x_{i}$ if $x_{i} \in P$ is minimal, but not maximal;
\item
$\Psi(x_{i}) = x_{i}$ if $x_{i} \in P$ is maximal;
\item
Let $x_{i}$ be neither minimal nor maximal.
If there is a unique saturated chain of the form $x = x_{i_{0}} > x_{i_{1}} > \cdots > x_{i_{k}}$,
where $x_{i_{k}}$ is a minimal element of $P$, then
\[
\Psi(x_{i}) = 1 - x_{i_{0}} - x_{i_{1}} - \cdots - x_{i_{k}};
\]
\item
Let $x_{i}$ be neither minimal nor maximal.
If there exist at least two saturated chains of the form
$x_{i} = x_{i_{0}} > x_{i_{1}} > \cdots > x_{i_{k}}$,
where $x_{i_{k}}$ is a minimal element of $P$, and if
there is a unique saturated chain of the form $x_{i} = x_{j_{0}} < x_{j_{1}} < \cdots < x_{j_{\ell}}$,
where $x_{j_{\ell}}$ is a maximal element of $P$, then
\[
\Psi(x_{i}) = x_{i} + x_{j_{1}} + \cdots + x_{j_{\ell}}.
\]
\end{itemize}
It is routine work to show that if $\Fc$ is a facet of $\Oc(P)$, then $\Psi(\Fc)$ is a facet of
$\Cc(P)$. We will prove this claim with the help of Example \ref{eg}. 

In fact, there are three types of facets for $\Oc(P)$:
\begin{enumerate}
\item [1)] a minimal element $x\le 1$;
\item [2)] a maximal element $y\ge 0$;
\item [3)] a cover relation $x\le y$ if $x$ covers $y$ in $P$.
\end{enumerate}
There are two types of facets for $\Cc(P)$: 
\begin{enumerate}
\item [1')] for each element in the poset $x\ge 0$;
\item [2')] each maximal chain $\sum_{i\in C}x_i\le 1$.
\end{enumerate}
 In Example \ref{eg}, $x_1\le 1$ is mapped to $1-x_1\le 1$, which is $x_1\ge 0$. For type 3) facets $x\le y$ of $\Oc(P)$, there are three cases. For any $x\in P$, if there is a unique saturated chain starting at $x$ going down to a minimal element, we call $x$ a \emph{down element}, otherwise, if there exists at two such chains, we call $x$ an \emph{up element}. Then there are two cases for facets of the form $x\le y$ of $\Oc(P)$.
 \begin{enumerate}
 \item [a)]Both $x$ and $y$ are down elements, then this facet is sent to 1') facet of $\Cc(P)$: $x\ge 0$. In Example \ref{eg}, $x_2\le x_1$ is mapped to $ 1-x_1-x_2\le 1-x_1$, which is $x_2\ge 0$.
 \item [b)]Both $x$ and $y$ are up elements, then this facet is sent to 1') facet of $\Cc(P)$: $y\ge 0$. In Example \ref{eg}, $x_9\le x_7$ is mapped to $ x_{9}+x_{11}\le x_7+x_9+x_{11}$, which is $x_7\ge 0$.
 \item [c)]If $x$ is up and $y$ is down, then this facet is sent to a type 2') facet of $\Cc(P)$. In Example \ref{eg}, $x_7\le x_2$ is mapped to $x_7+x_9+x_{11}\le 1-x_1-x_2 $, which is $x_1+x_2+x_7+x_9+x_{11}\le 1 $.
 \end{enumerate}
 Hence $\Psi(\Oc(P)) = \Cc(P)$.  Thus $\Oc(P)$ and $\Cc(P)$ are affinely equivalent.
Moreover, since $\Psi(\ZZ^{n}) = \ZZ^{n}$ and since the volume of $\Oc(P)$ coincides with
that of $\Cc(P)$, it follows that $\Oc(P)$ and $\Cc(P)$ are unimodularly equivalent.
\, \, \, \, \, \, \, \, \, \, \, \, \, \, \, \, \, \, \, \, \, \, \, \, \, \, \, \, \,
\end{proof}

\begin{Example}
{\em
\label{eg}Consider the following poset.
$$ 
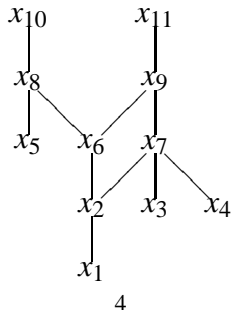
\xy 0;/r.2pc/:
(0,0)*{x_1}="1";
 (0,10)*{x_2}="2";
 (10,10)*{x_3}="3";
 (20,10)*{x_4}="4";
 (-10,20)*{x_5}="5";
 (0,20)*{x_6}="6";
 (10,20)*{x_7}="7";
 (-10,30)*{x_8}="8";
 (10,30)*{x_9}="9";
 (-10,40)*{x_{10}}="10";
 (10,40)*{x_{11}}="11";
   "10"; "8"**\dir{-};
    "5"; "8"**\dir{-};
     "2"; "6"**\dir{-};
      "1"; "2"**\dir{-};
    "7"; "3"**\dir{-};
     "7"; "9"**\dir{-};
      "11"; "9"**\dir{-};
    "4"; "7"**\dir{-};
     "2"; "7"**\dir{-};
      "6"; "9"**\dir{-};
    "6"; "8"**\dir{-};
\endxy$$

\hspace{6.3cm}
{\rm \bf \small Figure\,\,2}

\noindent
Both $\Oc(P)$ and $\Cc(P)$ have $17$ facets. Facets of $\Oc(P)$ are
$$\begin{array}{cccccc}
  x_1\le 1, & x_3\le 1, & x_4\le 1, & x_5\le 1, & x_{10}\ge 0, & x_{11}\ge 0, \\
  x_1\ge x_2, & x_2\ge x_7, & x_3\ge x_7, & x_4\ge x_7, & x_2\ge x_6, & x_5\ge x_8, \\
  x_6\ge x_8, & x_6\ge x_9, & x_7\ge x_9, & x_8\ge x_{10}, & x_9\ge x_{11}, &
\end{array}$$
Facets of $\Cc(P)$ are
$$\begin{array}{cccccc}
  x_1\ge 0, & x_2\ge 0, & x_3\ge 0, & x_4\ge 0, & x_{5}\ge 0, & x_{6}\ge 0, \\
  x_{7}\ge 0, & x_{8}\ge 0, & x_{9}\ge 0, & x_{10}\ge 0, & x_{11}\ge 0, & \\
\end{array}$$
$$\begin{array}{ccc}
  x_1+x_2+x_7+x_9+x_{11}\le 1, & x_3+x_7+x_9+x_{11}\le 1, & x_4+x_7+x_9+x_{11}\le 1, \\
  x_1+x_2+x_6+x_8+x_{10}\le 1, & x_5+x_8+x_{10}\le 1, & x_1+x_2+x_6+x_9+x_{11}\le 1. \\
\end{array}$$
Here is the map $\Psi$ defined in Theorem \ref{unimodular}:
$$\begin{array}{cccccc}
  x_1\mapsto 1-x_1, & x_2\mapsto 1-x_1-x_2, & x_3\mapsto 1-x_3, & x_4\mapsto 1-x_4, \\
   x_5\mapsto 1-x_5, & x_6\mapsto 1-x_6-x_2-x_1, & x_7\mapsto x_7+x_9+x_{11}, & x_8\mapsto x_8+x_{10}, \\
    x_9\mapsto x_9+x_{11}, & x_{10}\mapsto x_{10}, & x_{11}\mapsto x_{11}. &
\end{array}
$$
}
\end{Example}

\begin{Corollary}
\label{Corollary}
Given a finite poset $P$, the following conditions are equivalent:
\begin{enumerate}
\item[(i)] $\Oc(P)$ and $\Cc(P)$ are unimodularly equivalent;
\item[(ii)] $\Oc(P)$ and $\Cc(P)$ are affinely equivalent;
\item[(iii)] $\Oc(P)$ and $\Cc(P)$ have the same $f$-vector {\rm (\cite[p. 12]{HibiRedBook})};
\item[(iv)] The number of facets of $\Oc(P)$ is equal to that of $\Cc(P)$;
\item[(v)] The poset of Figure $1$ of Theorem \ref{characterization} does not appear
as a subposet of $P$.
\end{enumerate}
\end{Corollary}

\begin{Conjecture}
Let $P$ be a finite poset with $|P| = d > 1$.
Let $f(\Oc(P)) = (f_{0}, f_{1}, \ldots, f_{d-1})$ denote the $f$-vector of $\Oc(P)$ and
$f(\Cc(P)) = (f'_{0}, f'_{1}, \ldots, f'_{d-1})$ the $f$-vector of $\Cc(P)$.  Then
\begin{enumerate}
\item[(a)] $f_{i} \leq f'_{i}$ for all $1 \leq i \leq d-1$.
\item[(b)] If $f_{i} = f'_{i}$ for some $1 \leq i \leq d-1$, then $\Oc(P)$ and $\Cc(P)$
are unimodularly equivalent.
\end{enumerate}
\end{Conjecture}

{}


\begin{thebibliography}{}
\bibitem{Hibi}
T. Hibi,
Distributive lattices, affine semigroup rings and algebras with straightening laws,
{\em in} ``Commutative Algebra and Combinatorics''
(M.\ Nagata and H.\ Matsumura, Eds.),
Advanced Studies in Pure Math., Volume 11, North--Holland, Amsterdam, 1987, pp. 93 -- 109.
\bibitem{HibiRedBook}
T. Hibi,
``Algebraic combinatorics on convex polytopes,''
Carslaw Publications, Glebe, N.S.W., Australia, 1992.
\bibitem{TN}
T. Hibi and N. Li,
Chain polytopes and algebras with straightening laws,
arXiv:1207.2538.
\bibitem{Stanley}
R. Stanley,
Two poset polytopes, {\em Discrete Comput. Geom.} {\bf 1} (1986), 9 -- 23.
\bibitem{StanleyEC}
R. Stanley,
``Enumerative Combinatorics, Volume I,''
Second Ed., Cambridge University Press, Cambridge, 2012.
\end{thebibliography}
\end{document}